\documentclass[a4paper,11pt,reqno]{amsart}

\usepackage{dsfont}
\usepackage{amssymb}
\usepackage[all]{xy}

\chardef\bslash=`\\ % p. 424, TeXbook

\numberwithin{equation}{section}

\newtheorem{theorem}{Theorem}[section]
\newtheorem{corollary}[theorem]{Corollary}
\newtheorem{lemma}[theorem]{Lemma}
\newtheorem{proposition}[theorem]{Proposition}

\theoremstyle{remark}
\newtheorem{remark}[theorem]{Remark}

\theoremstyle{definition}
\newtheorem{definition}[theorem]{Definition}

\newcommand\bp{\begin{proof}}
\newcommand\ep{\end{proof}}

\newcommand\aaa{\mathfrak a}
\newcommand\bb{\mathfrak b}
\newcommand\dd{\mathfrak d}
\newcommand\mm{\mathfrak m}
\newcommand\pp{\mathfrak p}
\newcommand\qq{\mathfrak q}

\newcommand{\C}{\mathbb C}
\newcommand{\F}{\mathbb F}
\newcommand{\K}{\mathbb K}
\newcommand{\N}{\mathbb N}
\newcommand{\Q}{\mathbb Q}
\newcommand{\R}{\mathbb R}
\newcommand{\Z}{\mathbb Z}

\newcommand\A{\mathbb{A}}

\newcommand\ak{{\mathbb A}}
\newcommand\aks{{\mathbb A}_S}
\newcommand\akf{{\mathbb A}_{f}}
\newcommand{\ab}{\operatorname{ab}}
\newcommand\chr{\mathds 1}
\newcommand\CT{\C_\infty\{\tau\}}

\newcommand\oa{{\mathcal O}_{\mathbb A}}
\newcommand\oas{{\mathcal O}_{\mathbb A,S}}
\newcommand\OO{{\mathcal O}}
\newcommand\ohs{{\hat{\OO}^*}}

\newcommand\ord{\operatorname{ord}}
\newcommand\Spl{\operatorname{Spl}}
\newcommand{\imo}{I_\mm}
\newcommand{\Pic}{\operatorname{Pic}}
\newcommand{\pplus}{\PP^+}

\newcommand{\pplusmo}{\PP_\mm^+}
\newcommand{\picplusmo}{\Pic_\mm^+(\OO)}

\newcommand{\tor}{\operatorname{tor}}

\newcommand\CC{\mathcal C}
\newcommand\D{\mathcal D}
\newcommand{\GG}{\mathcal G}

\newcommand{\RR}{\mathcal R}
\newcommand{\PP}{\mathcal P}

\newcommand\eps{\varepsilon}

\newcommand\Gal{\operatorname{Gal}}

\newcommand{\id}{\operatorname{id}}

\newcommand{\sgn}{\operatorname{sgn}}

\newcommand\enu[1]{\smallskip\newline\makebox[5mm][l]{\rm(#1)}}

\begin{document}

\title[Bost-Connes systems]
{Bost-Connes systems associated with function fields}

\author[S. Neshveyev]{Sergey Neshveyev}
\address{Department of Mathematics, University of Oslo,
P.O. Box 1053 Blindern, N-0316 Oslo, Norway}

\email{sergeyn@math.uio.no}

\author[S. Rustad]{Simen Rustad}
\address{Department of Mathematics, University of Oslo,
P.O. Box 1053 Blindern, N-0316 Oslo, Norway}

\email{simenru@math.uio.no}

\thanks{Supported by the Research Council of Norway.}



\begin{abstract}
With a global function field $K$ with constant field $\F_q$, a finite set $S$ of primes in $K$ and an abelian extension $L$ of $K$, finite or infinite, we associate a C$^*$-dynamical system. The systems, or at least their underlying groupoids, defined earlier by Jacob using the ideal action on Drinfeld modules and by Consani-Marcolli using commensurability of $K$-lattices are isomorphic to particular cases of our construction. We prove a phase transition theorem for our systems and show that the unique KMS$_\beta$-state for every $0<\beta\le1$ gives rise to an ITPFI-factor of type III$_{q^{-\beta n}}$, where $n$ is the degree of the algebraic closure of $\F_q$ in $L$. Therefore for $n=+\infty$ we get a factor of type III$_0$. Its flow of weights is a scaled suspension flow of the translation by the Frobenius element on $\Gal(\bar \F_q/\F_q)$.
\end{abstract}

\date{December 26, 2011; minor corrections March 11, 2014}

\maketitle

\bigskip

%%%%%%%%%%%%%%%%%%%%%%%%%%%%%%%%%
\section*{Introduction}
%%%%%%%%%%%%%%%%%%%%%%%%%%%%%%%%%

In 1995 Bost and Connes~\cite{bos-con} constructed a quantum statistical mechanical system with remarkable connections to number theory. Among other properties, this system has phase transition with spontaneous symmetry breaking: the symmetry group of the unique equilibrium state for large temperatures contains $\Gal(\Q^{\ab}/\Q)$, while for small temperatures~$T$ the extremal equilibrium states have partition function $\zeta(T^{-1})$ and no symmetries in the Galois group. After numerous and only partially successful attempts by a number of people, the accepted correct analogue of this system for an arbitrary number field was defined by Ha and Paugam in 2005~\cite{ha-pa}. Soon after that analogues of the Bost-Connes system for function fields were proposed by Jacob~\cite{Ja} and Consani-Marcolli~\cite{CM}. Their constructions are not directly comparable, since Jacob works within the traditional C$^*$-algebraic setting, while Consani and Marcolli develop quantum statistical mechanics over fields of non-zero characteristic. In both cases, however, the algebras of observables are defined using groupoids, and these can be compared. Both constructions are motivated by the explicit class field theory for function fields. The standard formulation of the latter fixes a prime $\infty$ in a function field $K$ and describes either the maximal abelian extension $K^{\ab,\infty}$ of $K$ that is completely split at infinity, or a slightly larger extension $\K$ that is totally ramified over $K^{\ab,\infty}$ at infinity. The fields $\K$ and $K^{\ab,\infty}$ play the role of~$\Q^{\ab}$ in the constructions of Jacob and Consani-Marcolli, respectively. The present paper emerged from our attempt to answer the following two natural questions. Is there an analogue of the Bost-Connes system for $K$ corresponding to the maximal abelian extension rather than to $\K$ or $K^{\ab,\infty}$? What is the exact relation between the groupoids of Jacob and Consani-Marcolli?

\smallskip

In Section~\ref{sphase} we define a C$^*$-dynamical system $(A_{L,S},\sigma)$ for every abelian extension $L$ of~$K$ and a finite set $S$ of primes in $K$, and, following the strategy in \cite{LLN}, show that its equilibrium states have the expected properties. The case $L=K^{\ab}$ and $S=\emptyset$ is our answer to the first question. The definition of  $(A_{L,S},\sigma)$ looks straightforward, if one compares it with the system of Ha and Paugam as presented in~\cite{LLN}. A ``new'' simple idea is to consider all possible abelian extensions instead of a fixed one and to allow removing a finite set of primes. This brings a lot of flexibility into the analysis of the systems and allows us to use the standard number-theoretic strategy: reduce a proof of a given property to the situation where the extension $L/K$ is finite and $S$ contains all primes in~$K$ that ramify in $L$, in which case the system $(A_{L,S},\sigma)$ has a very transparent structure.

\smallskip

In Section~\ref{stype} we compute the type of the unique equilibrium state of $(A_{L,S},\sigma)$ at inverse temperature $\beta=1/T$ for $0<\beta\le1$. The spectrum of the corresponding modular operator consists of powers of $q^{\beta}$, where $q$ is the number of elements in the constant field of $K$. A natural guess is, therefore, that the type is III$_{q^{-\beta}}$. This would indeed be the case, if we could show that the centralizer of the state is a factor. It turns out that this is the case if and only if the extension $L/K$ is geometric. In the general case the type is III$_{q^{-\beta n}}$, where $q^n$ is the number of elements in the constant field of~$L$. After a number of reductions the proof boils down to a computation of the asymptotic ratio set for a sequence of measures. This computation, in turn, relies on a version of the Chebotarev density theorem. When $n=+\infty$, the type is III$_0$, and we show that the corresponding flow of weights is a suitably scaled suspension flow of the translation by the Frobenius element on $\Gal(\bar \F_q/\F_q)$.

\smallskip

In the remaining part of the paper we show that the systems of Jacob and Consani-Marcolli fit into our framework, with $L=\K$ or $L=K^{\ab,\infty}$ and $S=\{\infty\}$. In case of the Consani-Marcolli groupoid this is quite clear from results in~\cite{CM}. But Jacob's system requires some work.

In Section~\ref{sdrinfeld} we summarize results on Drinfeld modules that are needed later. In Section~\ref{sjacob} we then prove that Jacob's system is isomorphic to our system corresponding to $L=\K$ and $S=\{\infty\}$. The isomorphism is non-canonical, but any two isomorphisms in the class we define differ by an element of $\Gal(\K/K)$ acting on one of the system. Our approach to Bost-Connes systems therefore provides, in the case  $L=\K$ and $S=\{\infty\}$, an alternative route to results of Jacob and corrects his statement about the types of the equilibrium states in the critical region.

Returning to our second motivating question, the isomorphism established in Section~\ref{sjacob} shows that the groupoid of Consani-Marcolli is non-canonically isomorphic to the quotient of Jacob's groupoid by the action of $\Gal(\K/K^{\ab,\infty})$. Furthermore, a small modification of the construction of Consani-Marcolli leads to a groupoid that is isomorphic to Jacob's groupoid.

\smallskip

In Appendix we collect some general results on the ratio set of an equivalence relation that we need for the computations in Section~\ref{stype}.

\bigskip

%%%%%%%%%%%%%%%%%%%%%%%%%%%%%%%%%
\section{Systems associated with abelian extensions} \label{sphase}
%%%%%%%%%%%%%%%%%%%%%%%%%%%%%%%%%

Let $K$ be a global function field with constant field $\F_q$. For every prime $\pp$ in $K$ denote by~$K_\pp$ the corresponding completion of $K$, and by $\OO_\pp\subset K_\pp$ its maximal compact subring. Denote by~$\D$ the group of divisors of~$K$, by $\ak$ the adele ring of $K$, and by $\oa\subset \ak$ its maximal compact subring. So $\D$ is the free abelian group with generators $\pp$, $\ak$ is the restricted product $\prod'K_\pp$ of the fields~$K_\pp$ with respect to $\OO_\pp\subset K_\pp$, and $\oa=\prod_\pp\OO_\pp$. For a finite set $S$ of primes in $K$ denote by $\D_S\subset\D$ the subgroup consisting of divisors with support in the complement $S^c$ of $S$. Similarly, define $\ak_S=\prod'_{\pp\in S^c}K_\pp$ and $\oas=\prod_{\pp\in S^c}\OO_\pp$.

\smallskip

Let $L/K$ be an abelian extension, finite or infinite, and $S$ be a finite, possibly empty, set of primes in~$K$.
Consider the space
$$
X_{L,S}=\Gal(L/K)\times_{\oas^*}\ak_S.
$$
Here the action of $\oas^*$ on $\Gal(L/K)$ is defined using the Artin map $r_{L/K}\colon\ak^*\to\Gal(L/K)$. Thus $X_{L,S}$ is the quotient of $\Gal(L/K)\times\ak_S$ by the action $g(x,y)=(xr_{L/K}(g)^{-1},gy)$ of $\oas^*$.
Identify~$\D_S$ with $\aks^*/\oas^*$. Then the diagonal action of $\aks^*$ on $\Gal(L/K)\times\aks$, given again by $g(x,y)=(xr_{L/K}(g)^{-1},gy)$, defines an action of $\D_S$ on $X_{L,S}$. Put $Y_{L,S}=\Gal(L/K)\times_{\oas^*}\oas\subset X_{L,S}$. Consider the C$^*$-algebra
$$
A_{L,S}=\chr_{Y_{L,S}}(C_0(X_{L,S})\rtimes \D_S)\chr_{Y_{L,S}}.
$$
We can also write $A_{L,S}$ as the semigroup crossed product $C(Y_{L,S})\rtimes\D_S^+$, where $\D_S^+\subset\D_S$ is the subsemigroup of effective divisors.

The action of $\Gal(L/K)$ by translations on itself defines an action of $\Gal(L/K)$ on $X_{L,S}$, which in turn defines an action on $A_{L,S}$.

Define a dynamics $\sigma$ on $A_{L,S}$. It is easier to explain its extension to the multiplier algebra of the whole crossed product $C_0(X_{L,S})\rtimes \D_S$, which we continue to denote by $\sigma$. We put $\sigma_t(f)=f$ for $f\in C_0(X_{L,S})$ and $\sigma_t(u_D)=N(D)^{it}u_D$ for $D\in\D_S$, where $N(D)=q^{\deg D}$ is the norm of $D$.

\smallskip

Recall that a KMS-state for $\sigma$ at inverse temperature
$\beta\in\R$, or a $\sigma$-KMS${}_\beta$-state,  is a
$\sigma$-invariant state~$\varphi$ such that $\varphi(ab) =
\varphi(b\sigma_{i\beta}(a))$ for $a$ and $b$ in a set of
$\sigma$-analytic elements with dense linear span. A $\sigma$-invariant state $\varphi$ is called a ground state, if the holomorphic function $z\mapsto\varphi(a\sigma_z(b))$ is bounded on
the upper half-plane for $a$ and $b$ in a set of $\sigma$-analytic
elements spanning a dense subspace. If a state $\varphi$ is a
weak$^*$ limit point of a sequence of states $\{\varphi_n\}_n$ such
that $\varphi_n$ is a $\sigma$-KMS$_{\beta_n}$-state and
$\beta_n\to+\infty$ as $n\to\infty$, then $\varphi$ is a ground
state. Such ground states are called
$\sigma$-KMS$_\infty$-states.

\begin{theorem} \label{tphase}
For the system $(A_{L,S},\sigma)$ we have:
\enu{i} for $\beta<0$ there are no KMS$_\beta$-states;
\enu{ii} for every $0< \beta \leq 1$ there is a
unique KMS$_\beta$-state; \enu{ii} for every  $1< \beta<\infty$ the
extremal KMS$_\beta$-states are indexed by the points of the subset $$Y^0_{L,S}=  \Gal(L/
K)\times_{\oas^*}\oas^* \cong \Gal(L / K)$$ of $Y_{L,S}$, with the state
corresponding to $w\in Y^0_{L,S}$ given by
\begin{equation} \label{eKMSext}
\varphi_{\beta,w} (\chr_{Y_{L,S}}fu_D\chr_{Y_{L,S}}) = \frac{\delta_{D,0}}{\zeta_{K,S}(\beta) }\sum_{D'\in
\D^+_S} N(D')^{-\beta} f(D'w),
\end{equation}
where $\zeta_{K,S}(\beta)=\sum_{D\in\D^+_S} N(D)^{-\beta}$; furthermore, every extremal KMS$_\beta$-state $\varphi_{\beta,w}$ is of type I$_\infty$ and its partition function is $\zeta_{K,S}(\beta)$;
\enu{iv}  the extremal ground states are indexed by~$Y^0_{L,S}$, with the state
corresponding to $w\in Y^0_{L,S}$ given by $\varphi_{\infty,w}(\chr_{Y_{L,S}}fu_D\chr_{Y_{L,S}})=\delta_{D,0}f(w)$, and
all ground states are KMS$_\infty$-states.
\end{theorem}

\bp Since the proof is essentially identical to that of \cite[Theorem~2.1]{LLN}, we will only describe the main steps. Along the way we will make a few remarks that will be important later.

First, one shows that classifying KMS$_\beta$-states for $\beta\in\R\setminus\{0\}$ is equivalent to classifying measures~$\mu$ on $X_{L,S}$ such that $\mu(Y_{L,S})=1$ and $\mu(DZ)=N(D)^{-\beta}\mu(Z)$ for any Borel set $Z\subset X_{L,S}$ and any divisor $D\in\D_S$. Namely, the state corresponding to $\mu$ is given by the composition of the canonical conditional expectation $A_{L,S}\to C(Y_{L,S})$ with the state on $C(Y_{L,S})$ defined by the measure $\mu$. Since $DY_{L,S}\subset Y_{L,S}$ for $D\in\D^+_S$, there are no such measures for $\beta<0$.

Next, one checks that if $\mu$ is a measure defining a KMS$_\beta$-state and $\beta>1$, then $Y^0_{L,S}$ is a fundamental domain (modulo a set of measure zero) for the action of $\D_S$ on $(X_{L,S},\mu)$, so the measure is completely determined by its restriction to $Y^0_{L,S}$. The case $\beta=+\infty$ is dealt with in a similar manner.

The most interesting and non-trivial region is $0<\beta\le1$. Let us see that a measure $\mu_\beta$ defining a KMS$_\beta$-state indeed exists. We define $\mu_\beta$ as the image of the measure
$$
\lambda_{L/K}\times\prod_{\pp\in S^c}\mu_{\beta,\pp}
$$
on $\Gal(L/K)\times\aks$ under the quotient map $\Gal(L/K)\times\aks\to X_{L,S}$, where $\lambda_{L/K}$ is the normalized Haar measure on $\Gal(L/K)$ and $\mu_{\beta,\pp}$ is the measure on $K_\pp$ defined as follows. The measure $\mu_{1,\pp}$ is the Haar measure on $K_\pp$ such that $\mu_{1,\pp}(\OO_\pp)=1$. The measure $\mu_{\beta,\pp}$ is absolutely continuous with respect to $\mu_{1,\pp}$ and
$$
\frac{d\mu_{\beta,\pp}}{d\mu_{1,\pp}}(a)=\frac{1-N(\pp)^{-\beta}}{1-N(\pp)^{-1}}\|a\|^{\beta-1}_\pp.
$$

Assume now that $\mu$ is a measure defining a KMS$_\beta$-state for some $0<\beta\le1$. In order to see that $\mu=\mu_\beta$, it is convenient to make two reductions.

Given an intermediate finite extension $E$ of $K$ we can identify $X_{E,S}$ with $X_{L,S}/\Gal(L/E)$. Since any continuous function with support in $Y_{L,S}$ can be approximated by a $\Gal(L/E)$-invariant function with support in $Y_{L,S}$ for sufficiently large $E$, it follows that to prove that $\mu=\mu_\beta$ it suffices to consider the systems corresponding to finite extensions.

For the second reduction take a finite set $S'$ of primes in $K$ containing $S$. The subset
$$
X_{L,S,S'}=\Gal(L/K)\times_{\oas^*}\left(\prod_{\pp\in S'\setminus S}\OO^*_\pp\times\ak_{S'}\right)
$$
of $X_{L,S}$ is $\D_{S'}$-invariant, it is a fundamental domain (modulo a set of measure zero) for the action of the group generated by primes in $S'\setminus S$ on $(X_{L,S},\mu)$, and the measure of
$$
Y_{L,S,S'}=\Gal(L/K)\times_{\oas^*}\left(\prod_{\pp\in S'\setminus S}\OO^*_\pp\times\OO_{\A,S'}\right)
$$
is equal to $\prod_{\pp\in S'\setminus S}(1-N(\pp)^{-\beta})$. It follows that classifying measures $\mu$ on $X_{L,S}$ defining KMS$_\beta$-states is the same as classifying measures $\nu$ on $X_{L,S,S'}$ such that $\nu(Y_{L,S,S'})=\prod_{\pp\in S'\setminus S}(1-N(\pp)^{-\beta})$ and $\nu(DZ)=N(D)^{-\beta}\nu(Z)$ for any Borel set $Z\subset X_{L,S,S'}$ and any divisor $D\in\D_{S'}$. But $X_{L,S,S'}$ can be identified with $X_{L,S'}$. Therefore, if we have uniqueness of KMS$_\beta$-states for~$S'$, we have it also for $S$.

To summarize, in proving the uniqueness of KMS$_\beta$-states for $0<\beta\le1$ we may assume that the extension $L/K$ is finite and the set $S$ is as large as we want. In particular, we may assume that $S$ contains all primes in $K$ that ramify in $L$. In this case the kernel of the Artin map $\aks^*\to\Gal(L/K)$ contains $\oas^*$, so this map factors through $\D_S$,
$$
X_{L,S}=\Gal(L/K)\times\aks/\oas^*
$$
and the action of $\D_S$ on $X_{L,S}$ is diagonal. In order to prove that $\mu=\mu_\beta$ one then computes the projection~$P$ in $L^2(Y_{L,S},d\mu)$ onto the subspace of $\D^+_S$-invariant functions, which turns out to consist only of constants. Consider a continuous function on $Y_{L,S}$ that is supported on $Y_{L,S,S'}\cong Y_{L,S'}=\Gal(L/K)\times\OO_{\ak,S'}/\OO^*_{\ak,S'}$ for some $S'\supset S$ and whose value at $a\in Y_{L,S,S'}$ depends only on the coordinate of $a$ in~$\Gal(L/K)$. It can be shown that if $f$ is defined by a nontrivial character $\chi$ of $\Gal(L/K)$, then the $L^2$-norm of $Pf$ is not larger than
$$
\prod_{\pp\in {S'}^c}\left|\frac{1-N(\pp)^{-\beta}}{1-\chi(\pp)N(\pp)^{-\beta}}\right|.
$$
Since the Artin $L$-function $L(s,\chi)$ does not have a pole at $s=1$, the above product diverges to zero (for $0<\beta\le1$). Hence $Pf=0$, so that $\int fd\mu=0=\int fd\mu_\beta$. From this it is easy to deduce that $\mu=\mu_\beta$.
\ep

The case $\beta=0$ is special, as then there are KMS-states that do not factor through the conditional expectation $A_{L,S}\to C(Y_{L,S})$. In order to classify them, consider the subfield $L^{un}_S\subset L$ such that $\Gal(L/L^{un}_S)={r_{L/K}(\oas^*)}$. This is the maximal subextension of $L/K$ unramified at all primes in~$S^c$. Let $\D_{L,S}$ be the kernel of the Artin homomorphism $\D_S\to\Gal(L^{un}_S/K)$, and $\D_{L,S}^0\subset \D_{L,S}$ be the subgroup of divisors of degree zero. Denote by $\mu_0$ the unique $\Gal(L/K)$-invariant measure on~$Y_{L,S}$ concentrated on the image of $\Gal(L/K)\times\{0\}\subset \Gal(L/K)\times\oas$ in $Y_{L,S}$.

\begin{proposition}
There is a one-to-one correspondence between extremal KMS$_0$-states on $A_{L,S}$ and characters of $\D^0_{L,S}$. Namely, the state $\tau_\chi$ corresponding to a character $\chi$ is given by
$$
\tau_{\chi} (\chr_{Y_{L,S}}fu_D\chr_{Y_{L,S}}) =\begin{cases}\chi(D)\int fd\mu_0,&\text{if}\ \ D\in \D^0_{L,S},\\0,&\text{otherwise}.\end{cases}
$$
\end{proposition}

\bp We will use a general classification result from~\cite{N3}, although the present case is not difficult to deal with directly. By \cite[Theorem~1.3]{N3} there is a one-to-one correspondence between KMS$_0$-states on $A_{L,S}$ and pairs $(\mu,\{\tau_x\}_{x})$ consisting of a $\D_S$-invariant measure $\mu$ on $X_{L,S}$ with $\mu(Y_{L,S})=1$ and a $\mu$-measurable field of states $\tau_x$ on $C^*(G_x)$, where $G_x\subset \D_S$ is the stabilizer of $x$, such that $\tau_{Dx}=\tau_x$ for $\mu$-a.e.~$x$ and $\tau_x$ factors through the canonical conditional expectation $C^*(G_x)\to C^*(G_x^0)$, where $G^0_x\subset G_x$ is the subgroup of divisors of degree zero. Assume $(\mu,\{\tau_x\}_{x})$ is such a pair. Since the intersection of the sets $DY_{L,S}$, $D\in\D_S$, coincides with the image $Z$ of $\Gal(L/K)\times\{0\}$ in $Y_{L,S}$, the measure $\mu$ is concentrated on $Z$. The set $Z$ can be identified with $\Gal(L/K)/{r_{L/K}(\oas^*)}=\Gal(L^{un}_S/K)$. Since the image of $\D_S$ in $\Gal(L^{un}_S/K)$ is dense, the Haar measure on $\Gal(L^{un}_S/K)$ is the unique $\D_S$-invariant measure. Therefore $\mu=\mu_0$ and the action of $\D_S$ on $(X_{L,S},\mu)$ is ergodic. It follows that the field $\{\tau_x\}_{x}$ is essentially constant. The stabilizer of every point in $Z$ is $\D_{L,S}$. Therefore we conclude that there is a one-to-one correspondence between KMS$_0$-states on $A_{L,S}$ and states on~$C^*(\D^0_{L,S})$. Then extremal KMS$_0$-states correspond to characters of $\D^0_{L,S}$.
\ep

\bigskip
\section{Types of KMS-states in the critical region}\label{stype}

We continue to use the notation of the previous section, so $L/K$ is an abelian extension of a global function field $K$ and $S$ is a finite set of primes in $K$. Take $0<\beta\le1$ and denote by $\varphi_\beta$ the unique $\sigma$-KMS$_\beta$-state on $A_{L,S}$.

\begin{theorem} \label{ttype}
Let $\F_{q^n}$, $n\in\N\cup\{+\infty\}$, be the algebraic closure of the constant field $\F_q\subset K$ in $L$. Then the von Neumann algebra $\pi_{\varphi_\beta}(A_{L,S})''$ is an injective factor of type III$_{q^{-\beta n}}$.
\end{theorem}

Since $\varphi_\beta$ is extremal, the von Neumann algebra $\pi_{\varphi_\beta}(A_{L,S})''$ is a factor. It is the reduction of the von Neumann algebra
$L^\infty(X_{L,S},\mu_\beta)\rtimes\D_S$ by the projection $\chr_{Y_{L,S}}$, where $\mu_\beta$ is the measure on $X_{L,S}$ defined in Section~\ref{sphase}.
Since $\chr_{Y_{L,S}}$ is a full projection in $A_{L,S}$, the von Neumann algebra $L^\infty(X_{L,S},\mu_\beta)\rtimes\D_S$ is also a factor, hence the action of $\D_S$ on $(X_{L,S},\mu_\beta)$ is ergodic. The equivalent formulation of the above theorem is therefore that this action (more precisely, the corresponding orbit equivalence relation) is of type III$_{q^{-\beta n}}$, see Appendix~\ref{a1}. Note that the injectivity of the factor $L^\infty(X_{L,S},\mu_\beta)\rtimes \D_S$ is obvious, since $\D_S$ is abelian.

\smallskip

Our computation of the ratio set will rely on the following version of the Chebotarev density theorem for function fields, see \cite[Proposition~6.4.8]{FJ}.

\begin{theorem}
Let $K$ be a function field with constant field $\F_q$, $L$ a finite Galois extension of $K$ with constant field $\F_{q^n}$, $\CC$ a conjugacy class in $\Gal(L/K)$ consisting of $c$ elements. Let $0\le a<n$ be such that the restriction of every element in $\CC$ to $\F_{q^n}$ is the $a$-th power of the Frobenius automorphism of~$\F_{q^n}/\F_q$. Then every prime $\pp$ such that $(\pp,L/K)=\CC$ has the property $\deg\pp\equiv a\mod n$, and
$$
\#\{\pp\mid (\pp,L/K)=\CC\ \text{and}\ \deg\pp=kn+a\}=\frac{cn}{(kn+a)[L:K]}q^{kn+a}+O(q^{kn/2})\ \ \text{as}\ \ k\to+\infty.
$$
\end{theorem}

We are interested in the case when the extension $L/K$ is abelian. Denote by $\Spl(L/K)$ the set of primes in $K$ that are completely split in $L$.

\begin{corollary} \label{cdistrib}
Assume $L/K$ is a finite abelian extension and $S$ contains the set of primes in $K$ that ramify in $L$. Then the degree of every element in the kernel of the Artin map $\D_S\to\Gal(L/K)$ is divisible by $n$, and
$$
\#\{\pp\mid\pp\in \Spl(L/K)\ \text{and}\ \deg\pp=kn\}=\frac{1}{k[L:K]}q^{kn}+O(q^{kn/2})\ \ \text{as}\ \ k\to+\infty.
$$
\end{corollary}

Under the assumptions of the previous corollary consider the restricted product $X'_{L,S}$ of the spaces~$K_\pp/\OO_\pp^*$ with respect to $\OO_\pp/\OO_\pp^*$ over all $\pp\in S^c\cap\Spl(L/K)$, and consider the measure $\mu'_\beta=\prod_{\pp\in S^c\cap\Spl(L/K)}\mu_{\beta,\pp}$ on $X'_{L,S}$. Denote by $\D'_S$ the subgroup of $\D_S$ generated by primes in $S^c\cap\Spl(L/K)$.

\begin{lemma} \label{ltype}
The action of $\D'_S$ on $(X'_{L,S},\mu'_\beta)$ is of type III$_{q^{-\beta n}}$.
\end{lemma}

\bp For $L=K$ the result is proved in \cite[Lemma~4.5.1]{Ja}. The general case is similar, but uses Corollary~\ref{cdistrib} instead of the distribution of prime ideals. Note also that an analogous statement for the field $\Q$ of rational numbers and $L=K$ goes back to~\cite{bla}. For the reader's convenience we will nevertheless sketch a proof.

Since the degree of every prime $\pp\in S^c\cap\Spl(L/K)$ is divisible by $n$ and the measure $\mu'_\beta$ has the property $\mu'_\beta(D\,\cdot)=N(D)^{-\beta}\mu'_\beta$, the ratio set is contained in the set $\{0\}\cup\{q^{-\beta nk}\mid k\in\Z\}$. Therefore it suffices to show that $q^{-\beta n}$ lies in the ratio set. Instead of the orbit equivalence relation on $X'_{L,S}$ we may consider the relation induced on the subset $Y'_{L,S}=\prod_{\pp\in S^c\cap\Spl(L/K)}\OO_\pp^\times/\OO^*_\pp$. But this is exactly the relation discussed at the end of Appendix~\ref{a1}: two points are equivalent if and only if their coordinates coincide outside a finite set of primes. The measure $\mu_{\beta,\pp}$ on $\OO_\pp^\times/\OO^*_\pp$ is given by $\mu_{\beta,\pp}(\pi_\pp^k\OO^*_\pp)=N(\pp)^{-\beta k}(1-N(\pp)^{-\beta})$, where $\pi_\pp$ is a uniformizer in $\OO_\pp$.

Let $m_k$ be the number of primes in $S^c\cap\Spl(L/K)$ of degree $kn$. By Corollary~\ref{cdistrib} we have $m_k\sim[L:K]^{-1}q^{kn}/k$. In particular, for sufficiently large $k$ we have $m_{2k}<m_{2k+1}$. Hence for every prime $\pp$ of degree $2kn$ we can choose a prime $\pp'$ of degree $(2k+1)n$ in such a way that the map $\pp\mapsto\pp'$ is injective. Now as the sets $I_j$ required by the definition of the asymptotic ratio set we take the sets $\{\pp,\pp'\}$, and as the sets $K_j$ and $L_j$ we take the single-point sets $\{(\pi_\pp\OO^*_\pp,\OO^*_{\pp'})\}$ and  $\{(\OO^*_\pp,\pi_{\pp'}\OO^*_{\pp'})\}$, respectively. Since
$\sum_kq^{-2\beta kn}m_{2k}=\infty$, we conclude that $q^{-\beta n}$ belongs to the asymptotic ratio set.
\ep

\bp[Proof of Theorem~\ref{ttype}]
Assume first that the extension $L/K$ is finite. Since the action of $\D_S$ on $(X_{L,S},\mu_\beta)$ is ergodic, in computing the ratio set instead of the orbit equivalence relation on $X_{L,S}$ we may consider the relation induced on any subset of positive measure. In particular, for every $S'\supset S$ we may consider the subset $X_{L,S,S'}$ introduced in Section~\ref{sphase}. As we discussed there, it can be identified with $X_{L,S'}$, and the equivalence relation we get on $X_{L,S'}$ is exactly the one defined by the action of the group $\D_{S'}$. Therefore the type of the action of~$\D_S$ on~$(X_{L,S},\mu_\beta)$ does not depend on $S$. Hence we may assume that $S$ includes all primes that ramify in~$L$. Then $X_{L,S}=\Gal(L/K)\times\aks/\oas^*$.

Consider the subset $\{e\}\times\aks/\oas^*$ of $X_{L,S}$. The equivalence relation induced on it is the one given by the action of the kernel $G$ of the Artin map $\D_S\to\Gal(L/K)$ on $\aks/\oas^*$. By the first (obvious) part of Corollary~\ref{cdistrib} the degree of every element in $G$ is divisible by $n$. Hence the ratio set of the action of $G$ on $(\aks/\oas^*,\mu_\beta)$ is contained in $\{0\}\cup\{q^{-\beta nk}\mid k\in\Z\}$. On the other hand, by Lemma~\ref{ltype} the number~$q^{-\beta n}$ is contained in the ratio set of the action of $\D'_S\subset G$ on $(X'_{L,S},\mu'_\beta)$. Hence, by Proposition~\ref{pproduct}, it is contained in the ratio set of the action of $\D'_S$ on $(\aks/\oas^*,\mu_\beta)$. This proves the theorem when $L/K$ is finite.

Now consider an arbitrary abelian extension $L/K$. By Proposition~\ref{psequence} the nonzero part of the ratio set of the action of $\D_S$ on~$(X_{L,S},\mu_\beta)$ is equal to the intersection of the nonzero parts of the ratio sets of the actions of $\D_S$ on $(X_{L,S}/\Gal(L/E),\mu_\beta)$ for all finite intermediate extensions $E/K$. Since $X_{L,S}/\Gal(L/E)=X_{E,S}$, from the first part of the proof we conclude that if $n<+\infty$, then the nonzero part of the ratio set of the action of~$\D_S$ on $(X_{L,S},\mu_\beta)$ equals $\{q^{-\beta nk}\mid k\in\Z\}$, hence the action is of type III$_{q^{-\beta n}}$.

In the case $n=+\infty$ we can only conclude that the ratio set is either $\{1\}$ or $\{0,1\}$, so the factor $M_L=\pi_{\varphi_\beta}(A_{L,S})''$ is either semifinite or of type III$_0$. Since $M_K=M_L^{\Gal(L/K)}$, there exists a normal conditional expectation $M_L\to M_K$. As $M_K$ is of type III,
by \cite[Proposition~10.21]{Stra} it follows that~$M_L$ is also of type III, hence it is of type III$_0$.
\ep

For every $\lambda\in(0,1]$ there exists a unique injective factor of type III$_\lambda$ with separable predual. Therefore for $n<+\infty$ the above result completely describes the von Neumann algebra $\pi_{\varphi_\beta}(A_{L,S})''$. For type III$_0$ factors a complete invariant is the flow of weights. The Galois group $\Gal(L/K)$ acts on the flow of weights of $\pi_{\varphi_\beta}(A_{L,S})''$ and the quotient by $\Gal(L/E)$ gives the flow of weights of~$\pi_{\varphi_\beta}(A_{E,S})''$. Since the flows corresponding to finite extensions are periodic, it is not difficult to see that this is enough to completely describe the flow of weights of $\pi_{\varphi_\beta}(A_{L,S})''$, but we will do everything explicitly.

\smallskip

The flow of weights of a crossed product algebra, or of the von Neumann algebra of an equivalence relation, can be described as follows, see~\cite {CT}. Assume a countable group $\Gamma$ acts ergodically by non-singular transformations on a standard measure space $(X,\mu)$. Let $\lambda$ be the Lebesgue measure on $\R$. Define a new action of $\Gamma$ on $(\R\times X,\lambda\times\mu)$ by
$$
g(s,x)=\left(s-\log\frac{dg\mu}{d\mu}(gx),gx\right).
$$
Consider the measure theoretic quotient $\tilde X$ of $(\R\times X,\lambda\times\mu)$ by this action. So $\tilde X$ is a standard Borel space with a measure class $[\tilde\mu]$ such that $L^\infty(\tilde X,\tilde\mu)=L^\infty(\R\times X,\lambda\times\mu)^\Gamma$. The action of~$\R$ on~$\R\times X$ given by $t(s,x)=(s+t,x)$ induces an ergodic flow $\{F_t\}_{t\in\R}$ on $(\tilde X,\tilde\mu)$. This flow depends, up to isomorphism, only on the measure class of $\mu$ and the orbit equivalence relation $\RR$ on $X$ defined by the action of $\Gamma$. It is the flow of weights of the factor $W^*(\RR)$. It is known that the action of $\Gamma$ on $(X,\mu)$ is of type III$_0$ if and only if the flow is non-transitive, that is, there is no orbit of positive measure.

\smallskip

Returning to the Bost-Connes systems, consider an abelian extension $L/K$. As before, let $\F_{q^n}$, $n\in\N\cup\{+\infty\}$, be the algebraic closure of $\F_q$ in $L$.

Define a continuous map $X_{L,S}\to\Gal(\F_{q^n}/\F_q)$ as the composition of the quotient map
$$
X_{L,S}\to X_{L,S}/\Gal(L/\F_{q^n}K)=X_{\F_{q^n}K,S}=\Gal(\F_{q^n}K/K)\times\aks/\oas^*,
$$
where we have used that any finite constant field extension of $K$ is unramified at every prime, with the projection
$$
\Gal(\F_{q^n}K/K)\times\aks/\oas^*\to\Gal(\F_{q^n}K/K)=\Gal(\F_{q^n}/\F_q).
$$
Taking the direct product with the identity map on $\R$ we get a continuous map $$\R\times X_{L,S}\to\R\times\Gal(\F_{q^n}/\F_q).$$ This map becomes $\D_S$-equivariant, if we define an action of $\D_S$ on $\R\times\Gal(\F_{q^n}/\F_q)$ by
$$
D(s,g)=(s-\beta\log N(D),g\operatorname{res}_{\F_{q^n}}(r_{\F_{q^n}K/K}(D))^{-1}).
$$
Note that $\operatorname{res}_{\F_{q^n}}(r_{\F_{q^n}K/K}(D))$ is simply the Frobenius automorphism raised to the power $\deg D$.
Denote by $\lambda_n$ the normalized Haar measure on $\Gal(\F_{q^n}/\F_q)$. Then the map $\R\times X_{L,S}\to\R\times\Gal(\F_{q^n}/\F_q)$ gives us a $\D_S$-equivariant embedding
$$
L^\infty(\R\times\Gal(\F_{q^n}/\F_q),\lambda\times\lambda_n)\hookrightarrow L^\infty(\R\times X_{L,S},\lambda\times\mu_\beta).
$$

\begin{lemma} \label{lflow}
We have $L^\infty(\R\times X_{L,S},\lambda\times\mu_\beta)^{\D_S}=L^\infty(\R\times\Gal(\F_{q^n}/\F_q),\lambda\times\lambda_n)^{\D_S}$.
\end{lemma}

\bp It suffices to check that the subalgebras of $\Gal(L/E)$-invariant elements on both sides coincide for all intermediate finite extensions $E$ of $K$. Since $X_{L,S}/\Gal(L/E)=X_{E,S}$, this means that it is enough to prove the lemma for finite extensions.

Choose a finite set $S'$ which contains $S$ and all primes that ramify in $L$. Consider the set $X_{L,S,S'}$ introduced in Section~\ref{sphase}. It can be identified with $X_{L,S'}=\Gal(L/K)\times\ak_{S'}/\OO_{\ak,S'}^*$.  Consider the subset $\{e\}\times \ak_{S'}/\OO_{\ak,S'}^*$ of $X_{L,S'}$. We claim that if $f\in L^\infty(\R\times X_{L,S},\lambda\times\mu_\beta)^{\D_S}$, then the restriction of $f$ to $\R\times\{e\}\times \ak_{S'}/\OO_{\ak,S'}^*$ depends only on the first coordinate and therefore defines a function $f_1\in L^\infty(\R,\lambda)$, and the function $f_1$ is $\beta n\log q$-periodic.

Indeed, let $G$ be the kernel of the Artin map $\D_{S'}\to\Gal(L/K)$. By the proof of Theorem~\ref{ttype} the nonzero part of the ratio set of the action of $G$ on $\{e\}\times \ak_{S'}/\OO_{\ak,S'}^*$ coincides with the set of values of the Radon-Nikodym derivatives. By Proposition~\ref{pergodic} it follows that the subgroup $G_0\subset G$ of divisors of degree zero acts ergodically on $\{e\}\times \ak_{S'}/\OO_{\ak,S'}^*$. Since the group $G_0$ acts trivially on~$\R$, the $G_0$-invariant function $f$ on $\R\times\{e\}\times \ak_{S'}/\OO_{\ak,S'}^*$ depends only on the first coordinate and therefore defines a function $f_1\in L^\infty(\R,\lambda)$. Since $f$ is $G$-invariant and the homomorphism $\deg\colon G\to n\Z$ is surjective, the function $f_1$ is $\beta n\log q$-periodic.

The function $f_1$ defines, in turn, a $\D_S$-invariant function $f_2$ on $\R\times\Gal(\F_{q^n}/\F_q)\cong\R\times\Z/n\Z$ such that $f_2(s,0)=f_1(s)$. Namely, $f_2(s,m)=f_1(s-\beta m\log q)$. Since the $\D_S$-orbit of almost every point in $\R\times X_{L,S}$ intersects $\R\times\{e\}\times \ak_{S'}/\OO_{\ak,S'}^*$, every $\D_S$-invariant function is completely determined by its values on this set. Hence $f=f_2\in L^\infty(\R\times\Gal(\F_{q^n}/\F_q),\lambda\times\lambda_n)^{\D_S}$.
\ep

Assume now that $n=+\infty$, so that $\F_{q^n}=\bar\F_q$. We can identify the Galois group $\Gal(\bar \F_q/\F_q)$ with the group $\displaystyle\lim_{\leftarrow}\Z/k\Z=\hat\Z$. The action of $\D_S$ on $\R\times\hat\Z$ that we get is given by
$$
D(s,a)=(s-\beta\deg D\log q, a-\deg D).
$$
By rescaling the first coordinate we can instead consider the action $D(s,a)=(s-\deg D,a-\deg D)$. We can now formulate a refinement of Theorem~\ref{ttype} for $n=+\infty$.

\begin{theorem}
Assume the algebraic closure of $\F_q$ in $L$ is infinite. Then the von Neumann algebra $\pi_{\varphi_\beta}(A_{L,S})''$ is an ITPFI factor of type III$_0$. Its flow of weights is the flow on the compact group $(\R\times\hat\Z)/\Z$ defined by
$$
F_t(s,a)=\left(s+\frac{t}{\beta\log q},a\right).
$$
\end{theorem}

\bp As follows from Lemma~\ref{lflow} and the subsequent discussion, the flow of weights of the factor $L^\infty(X_{L,S},\mu_\beta)\rtimes\D_S$ has the form given in the formulation of the theorem. We already know that this factor is of type III$_0$, but this is also clear from our description of the flow, since the flow is obviously non-transitive. The flow is, however, approximately transitive~\cite{CW}, hence the factor is~ITPFI. Since $\pi_{\varphi_\beta}(A_{L,S})''$ is a reduction of $L^\infty(X_{L,S},\mu_\beta)\rtimes\D_S$, the same assertions hold for~$\pi_{\varphi_\beta}(A_{L,S})''$, moreover, the two factors are isomorphic.
\ep

Next we will determine the center of the centralizer of the state $\varphi_\beta$ on $\pi_{\varphi_\beta}(A_{L,S})''$. This centralizer is the reduction of the von Neumann algebra
$L^\infty(X_{L,S},\mu_\beta)\rtimes\D^0_S$ by the projection $\chr_{Y_{L,S}}$.
The center of $L^\infty(X_{L,S},\mu_\beta)\rtimes\D^0_S$ is $L^\infty(X_{L,S},\mu_\beta)^{\D^0_S}$. If the field $\F_q$ is not algebraically closed in $L$, then the ratio set of the action of~$\D_S$ on $(X_{L,S},\mu_\beta)$ is strictly smaller than the essential range of the Radon-Nikodym cocycle, so by Proposition~\ref{pergodic} the action of $\D^0_S$ on $(X_{L,S},\mu_\beta)$ cannot be ergodic.

Consider the map $X_{L,S}\to\Gal(\F_{q^n}/\F_q)$ constructed before Lemma~\ref{lflow}. It gives us an embedding of $L^\infty(\Gal(\F_{q^n}/\F_q),\lambda_n)$ into $L^\infty(X_{L,S},\mu_\beta)$.

\begin{proposition}
If $L/K$ is an abelian extension and $\F_{q^{n}}$, $n\in\N\cup\{+\infty\}$, is the algebraic closure of~$\F_q$ in~$L$, then
$
L^\infty(X_{L,S},\mu_\beta)^{\D^0_S}=L^\infty(\Gal(\F_{q^n}/\F_q),\lambda_n).
$
\end{proposition}

\bp As in the proof of Lemma~\ref{lflow}, we may assume that $L/K$ is finite. Consider the open subset
$$
Z=\Gal(L/\F_{q^n}K)\times_{\oas^*}\aks\subset X_{L,S}.
$$
The equivalence relation on $Z$ induced by the action of $\D_S$ on $X_{L,S}$ is the orbit equivalence relation defined by the action of the kernel $H$ of the map $\D_S\to\Gal(\F_{q^n}K/K)=\Gal(\F_{q^n}/\F_q)$. The group~$H$ is simply the subgroup of divisors in $\D_S$ of degree divisible by $n$. Since the action of $\D_S$ on $(X_{L,S},\mu_\beta)$ is of type III$_{q^{-\beta n}}$, by Proposition~\ref{pergodic} we conclude that the action of $\D^0_S\subset H$ on $(Z,\mu_\beta)$ is ergodic. It follows that any $\D^0_S$-invariant measurable subset of $X_{L,S}$ coincides, modulo a set of measure zero, with the union of translations of $Z$ by elements of $\Gal(L/K)$. The set $Z$ is nothing else but the pre-image of the unit element $e\in\Gal(\F_{q^n}/\F_q)$ under the map $X_{L,S}\to\Gal(\F_{q^n}/\F_q)$. Hence any $\D^0_S$-invariant measurable subset of $X_{L,S}$ is the pre-image of a subset of $\Gal(\F_{q^n}/\F_q)$.
\ep

We finish the section by making a few remarks about similarities and differences between the function field and number field~\cite{N2} cases. In both cases one can say that the key part is to compute the types of KMS-states for finite extensions (although this may not be immediately obvious from the argument for number fields in~\cite{N2}). If we have a finite extension~$L/K$ and an intermediate extension~$E/K$, then $\pi_{\varphi_\beta}(A_{E,S})''$ is a finite index subfactor of~$\pi_{\varphi_\beta}(A_{L,S})''$. For number fields the smallest factor $\pi_{\varphi_\beta}(A_{K,S})''$ is of type~III$_1$, and to prove this one only needs to know the distribution of prime ideals. This automatically implies that all other factors are of type III$_1$ as well~\cite{Loi}. For function fields the smallest factor $\pi_{\varphi_\beta}(A_{K,S})''$ is of type III$_{q^{-\beta}}$. Although this gives some restrictions on the type of $\pi_{\varphi_\beta}(A_{L,S})''$~\cite{Loi}, the exact value of the type still has to be computed, and the computation requires the Chebotarev density theorem.

\bigskip

%%%%%%%%%%%%%%%%%%%%%%%%%%%%%%%%%

\section{An extremely brief introduction to Drinfeld modules}\label{sdrinfeld}

In this section we will give a brief summary of the facts about Drinfeld modules
that are required to describe Jacob's system~\cite{Ja}. For a more thorough exposition see for
instance \cite{Goss}, \cite{Hay}, where proofs of the results of this section can be found.

\smallskip

Let $K$ be a global function field with constant field $\F_q$. Fix a distinguished prime $\infty$ in $K$. Denote by $\OO\subset K$ the subring of functions having no pole away from $\infty$. Let $\F_\infty=\F_{q^{d_\infty}}$ be the residue field of $K_\infty$. We will also consider it as a subfield of $K_\infty$. The valuation on~$K_\infty$ extends uniquely to $\bar K_\infty$, its algebraic closure. Let $\C_\infty$ be the completion of $\bar K_\infty$ with respect to this valuation. Then~$\C_\infty$ is algebraically closed and complete with respect to the valuation.

\smallskip

Let $\CT$ be the ring of twisted polynomials in $\tau$, with product satisfying the commutation relation $\tau a=a^q\tau$ for $a\in\C_\infty$. It
can be identified with the ring of $\F_q$-linear polynomial maps $\C_\infty\to\C_\infty$, with $\tau(x)=x^q$ and
multiplication given by composition of maps.
Define a map $D\colon\CT\to\C_\infty$ by sending a polynomial in $\tau$ to its constant term.

\begin{definition}
A Drinfeld module over $\C_\infty$ is a homomorphism $\phi\colon\OO \to\CT$, $a\mapsto\phi_a$, of $\F_q$-algebras
whose image is not contained in $\C_\infty$ and which satisfies $D(\phi_a) = a$. If the degree of $\phi_a$ (considered as a polynomial in $\tau$) is equal to $-d_\infty\ord_\infty(a)$ for all $a\in\OO^\times$, then $\phi$ is said to be of rank one.
\end{definition}

By a lattice in $\C_\infty$ we mean a finitely generated discrete $\OO$-submodule of $\C_\infty$. Rank-one lattices have the form $\xi\aaa$, where $\xi\in\C_\infty^\times$ and $\aaa$ is a nonzero ideal in $\OO$. Any such lattice $\Lambda$ defines a rank-one Drinfeld module as follows.

First define the exponential function associated to $\Lambda$ by
$$e_\Lambda(x) = x\prod_{\alpha\in \Lambda\setminus\{0\}} (1-x/\alpha).$$
This is an entire function on $\C_\infty$ that induces an $\F_q$-linear isomorphism $\C_\infty/\Lambda\cong\C_\infty$.
Next form a map
$$\phi_a^\Lambda(x) = ax \prod_{0\neq\alpha\in a^{-1}\Lambda/\Lambda} (1-x/e_\Lambda(\alpha))$$
for $a\in\OO$. The definition of $\phi^\Lambda_a$ is motivated by the identity $e_\Lambda(ax)=\phi^\Lambda_a(e_\Lambda(x))$. The map $\phi^\Lambda_a$ is $\F_q$-linear and polynomial, so $\phi^\Lambda_a$ is an element of $\CT$. Therefore $a\mapsto\phi_a^\Lambda\in\CT$ is a rank-one Drinfeld module.

This way we get a one-to-one correspondence between rank-one lattices in $\C_\infty$ and rank-one Drinfeld modules over $\C_\infty$. % We will denote by $\Lambda_\phi$ the lattice corresponding to a Drinfeld module $\phi$.

\smallskip

For a non-zero ideal $\aaa$ in $\OO$ and a Drinfeld module $\phi$, consider the left ideal of $\CT$ generated by the $\phi_a$ for $a\in\aaa$. This ideal can be shown to be principal, so there is a unique monic $\phi_\aaa\in\CT$ generating it. One can show that for each $b\in\OO$ there is a unique $\phi'_b\in\CT$ such that
$$\phi_\aaa\phi_b = \phi'_b\phi_\aaa.$$
The map $b\mapsto\phi'_b$ is a Drinfeld module, which we denote by $\aaa*\phi$. This gives an action of the nonzero ideals of $\OO$ on the set of Drinfeld modules. On the level of lattices, the Drinfeld module $\aaa*\phi^\Lambda$ corresponds to the lattice $D(\phi_\aaa)\aaa^{-1}\Lambda$.

\smallskip

We say that two lattices $\Lambda$ and $\Lambda'$ are homothetic, if $\Lambda'=\xi\Lambda$ for some $\xi\in\C_\infty^\times$. The corresponding Drinfeld modules are said to be isomorphic, and we get $\xi\phi^\Lambda\xi^{-1}=\phi^{\Lambda'}$.

\begin{definition}
A sign function on $K_\infty$ is a homomorphism $\sgn\colon K_\infty^\times\to\F^\times_\infty$ which coincides with the reduction modulo $\infty$ on $\OO^*_\infty$.
\end{definition}

From now on we fix a sign function. An element of $K_\infty^\times$ is called positive, if it belongs to the kernel of $\sgn$. Denote by $K^\times_+\subset K^\times$ the subgroup of positive elements in $K$.

Given a Drinfeld module $\phi$, write $\mu_\phi$ for the function mapping $a\in\OO^\times$ to the leading coefficient of $\phi_a$. We say that $\phi$ is normalized, if $\mu_\phi$ has image contained in~$\F_\infty^\times$, and that $\phi$ is $\sgn$-normalized, if there is some $\sigma\in\Gal(\F_\infty/\F_q)$ (depending on $\phi$) such that $\mu_\phi(a)=\sigma(\sgn(a))$ for all $a$.

If $\phi$ is a $\sgn$-normalized Drinfeld module, then by definition $\phi_{(a)}=\phi_a$ for every positive $a\in\OO^\times$. This in particular implies that the semigroup $\PP^+$ of principal ideals in $\OO$ with positive generators acts trivially on $\phi$. Let $I$ be the semigroup of nonzero ideals in $\OO$. Put $\Pic^+(\OO)=I/\PP^+$.

\begin{theorem}
Every rank-one Drinfeld module over $\C_\infty$ is isomorphic to a $\sgn$-normalized Drinfeld module. The action by ideals of $\OO$ preserves the set of $\sgn$-normalized rank-one Drinfeld modules and is transitive on this set, and if $\aaa*\phi=\bb*\phi$ for some ideals $\aaa$ and $\bb$ and a $\sgn$-normalized rank-one Drinfeld module $\phi$, then $\aaa$ and $\bb$ define the same class in $\Pic^+(\OO)$.
\end{theorem}
%\cite[Theorem 12.3, Theorem 13.8]{Hay}
In particular, the set of $\sgn$-normalized rank-one Drinfeld modules consists of
$$\#\Pic^+(\OO)=h(\OO)\frac{q^{d_\infty}-1}{q-1}=h(K)d_\infty\frac{q^{d_\infty}-1}{q-1}$$ elements, where $h(\OO)$ is the class number of $\OO$ and $h(K)$ is that of $K$.

\smallskip

Given a finite abelian extension $L/K$ and a finite prime $\pp$ in $K$ unramified in $L$, we write $\sigma_\pp$ for the Frobenius automorphism $(\pp,L/K)\in\Gal(L/K)$. More generally, we write $\sigma_\aaa$ for the automorphism corresponding to every ideal $\aaa$ in $\OO$ whose prime factors are unramified in $L$.

\smallskip

Let $\phi$ be a fixed $\sgn$-normalized rank-one Drinfeld module, $y\in\OO$ be any nonconstant element, and $H^+$ be the field generated over $K$ by the coefficients of $\phi_y$.

\begin{theorem}
The extension $H^+/K$ is finite, abelian and unramified at every finite prime $\pp$ in $K$. It is independent of the choice of $\phi$ and~$y$. The Artin map $I\to\Gal(H^+/K)$ defines an isomorphism $\Pic^+(\OO)\cong\Gal(H^+/K)$, and for any $a\in\OO$ and a nonzero ideal $\aaa\subset\OO$ we have
$$
\sigma_\aaa(\phi_a)=(\aaa*\phi)_a.
$$
\end{theorem}
%\cite[Proposition 14.4, Theorem 14.7]{Hay}

It follows that $H^+$ contains the Hilbert class field $H$ of $\OO$, which is the maximal abelian extension of $K$ that is unramified at every finite prime and completely split at infinity. The Galois group of~$H/K$ is $\Pic(\OO)=I/\PP$, where $\PP$ is the semigroup of nonzero principal ideals in $\OO$.

\smallskip

Given a $\sgn$-normalized rank-one Drinfeld module $\phi$ we can define an $\OO$-module structure on $\C_\infty$ using the maps $\phi_a$. For a nonzero proper ideal $\mm$ in $\OO$ denote by $\phi[\mm]\subset\C_\infty$ the set of $\mm$-torsion points with respect to this structure. If $\phi=\phi^\Lambda$, then the exponential map $e_\Lambda$ defines an $\OO$-module isomorphism $\mm^{-1}\Lambda/\Lambda\cong\phi[\mm]$.

Put $K_\mm=H^+(\phi[\mm])$. Denote by $\imo$ be the semigroup of ideals of $\OO$ which are relatively prime to~$\mm$, and let
$$\pplusmo=\{(a) \mid a\in \OO^\times,\ \sgn(a)=1,\ a\equiv 1 \mod\mm\}\subset\imo.$$
The quotient $\picplusmo=\imo/\pplusmo$ is called the narrow ray class group modulo $\mm$ relative to $\sgn$. Define also
$$
\PP_\mm=\{(a) \mid a\in \OO^\times,\ a\equiv 1 \mod\mm\}\subset \imo.
$$

\begin{theorem}
The extension $K_\mm/K$ is finite, abelian and unramified away from $\infty$ and the prime ideals dividing $\mm$.  It is independent of the choice of $\phi$.
The Artin map $\imo\to\Gal(K_\mm/K)$ defines an isomorphism $\picplusmo\cong\Gal(K_\mm/K)$, and for any $\aaa\in\imo$ and $\lambda\in\phi[\mm]$ we have
$$\sigma_\aaa(\lambda) = \phi_\aaa(\lambda).$$

The subfield $K^+_\mm\subset K_\mm$ of elements fixed by the semigroup $\PP_\mm$ is contained in $K_\infty$, so the extension $K^+_\mm/K$ is completely split at infinity. If we continue to denote by $\infty$ the extension of the place $\infty$ of~$K$ to $K^+_\mm$ defined by the inclusion $K^+_\mm\subset K_\infty$, then $K_\mm/K^+_\mm$ is totally ramified at $\infty$.
\end{theorem}
%\cite[Theorem 16.2 and following discussion]{Hay}

Clearly, the Galois group of $K_\mm^+/K$ is $\Pic_\mm(\OO)=\imo/\PP_\mm$. Note also that the Galois group of~$K_\mm/H^+$ is isomorphic to $(I_\mm\cap\PP^+)/\PP_\mm^+$. Since $\OO^*=\F_q^\times$, the only positive element in $\OO^*$ is~$1$, hence any ideal in $\PP^+$ has a unique positive generator. We therefore have a well-defined map $I_\mm\cap\PP^+\to(\OO/\mm)^*$, $(a)\mapsto a\mod \mm$ for positive $a$. It is not difficult to see that it is surjective, whence
\begin{equation} \label{egalois}
\Gal(K_\mm/H^+)\simeq (\OO/\mm)^*.
\end{equation}

Let $\K\subset\C_\infty$ be the union of the fields $K_\mm$ over non-zero proper ideals $\mm$ of $\OO$, and $K^{\ab,\infty}\subset K_\infty$ be the union of the fields $K_\mm^+$. These are abelian extensions of $K$ and
$$
\Gal(\K/K)\cong\lim_{\leftarrow}\Pic^+_\mm(\OO)\ \ \text{and}\ \ \Gal(K^{\ab,\infty}/K)\cong\lim_{\leftarrow}\Pic_\mm(\OO),
$$
where $\mm$ runs over the nonzero ideals of $\OO$ ordered by divisibility.
It follows that the Artin map $\akf^*\to\Gal(\K/K)$, where $\akf=\A_{\{\infty\}}$, defines isomorphisms
$$
\Gal(\K/K)\cong\akf^*/K_+^\times\ \ \text{and}\ \ \Gal(K^{\ab,\infty}/K)\cong\akf^*/K^\times.
$$
The field $K^{\ab,\infty}$ is the maximal abelian extension of $K$ in which $\infty$ splits completely.

\smallskip

Later we will need to know the constant fields of $\K$ and $K^{\ab,\infty}$ in order to calculate the types of KMS$_\beta$-states in the critical region.

\begin{lemma}\label{constantfields}
The constant fields of $K_\mm$ and $K_\mm^+$ are both equal to $\F_{q^{d_\infty}}$.
\end{lemma}
\begin{proof}
First consider $K_\mm^+$. On the one hand, the residue field of $K_\mm^+$ at infinity is $\F_{q^{d_\infty}}$, since $K_\mm^+$ is completely split at infinity. On the other hand, $K_\mm^+$ contains the Hilbert class field~$H$. Since $\F_{q^{d_\infty}}K/K$ is unramified at every prime and completely split at infinity, we have $\F_{q^{d_\infty}}K\subset H\subset K_\mm^+$. Therefore $\F_{q^{d_\infty}}$ is both the constant field of $K_\mm^+$ and the residue field of $K_\mm^+$ at infinity. Since $K_\mm/K_\mm^+$ is totally ramified at infinity, the residue field of $K_\mm$ at infinity is $\F_{q^{d_\infty}}$. Hence the constant field of~$K_\mm$ is also $\F_{q^{d_\infty}}$.
\end{proof}

\begin{corollary} \label{cconstantf}
The algebraic closures of $\F_q$ in $\K$ and $K^{\ab,\infty}$ are equal to $\F_{q^{d_\infty}}$.
\end{corollary}

Finally, note that the isomorphisms \eqref{egalois} imply that the Artin map $\akf^*\to\Gal(\K/K)$ induces an isomorphism
$$
\Gal(\K/H^+)\cong\ohs,
$$
where $\hat\OO=\OO_{\A,\{\infty\}}$. This is also clear from the isomorphisms $\Gal(H^+/K)\cong\Pic^+(\OO)\cong\akf^*/K^\times_+\ohs$ and $\Gal(\K/K)\cong\akf^*/K_+^\times$.

\bigskip

\section{Dynamical systems arising from Drinfeld modules}\label{sjacob}

Inspired by the work of Drinfeld \cite{Dri} and Hayes \cite{Hay0} on explicit class field theory for function fields, Jacob \cite{Ja} defined a dynamical system associated to a function field $K$. The main goal of this section is to show that Jacob's system fits into our framework. %This does however require us to review Jacob's construction.

\smallskip

For a $\sgn$-normalized rank-one Drinfeld module $\phi$, write $\phi(\C_\infty)^{\tor}$ for the set $\cup_\mm\phi[\mm]$ of torsion points of the $\OO$-action on $\C_\infty$ given by $a\xi = \phi_a(\xi)$. If $\phi=\phi^\Lambda$, then $e_\Lambda$ defines an $\OO$-module isomorphism $K\Lambda/\Lambda\cong\phi(\C_\infty)^{\tor}$. Let $X_\phi$ be the group of characters of $\phi(\C_\infty)^{\tor}$. Put $X=\bigsqcup_\phi X_\phi$, where the union is taken over the set of $\sgn$-normalized rank-one Drinfeld modules. Since the set of such modules is finite and each $X_\phi$ is a profinite group, $X$ is compact.

Define an action of $I$ on $X$ by letting an ideal $\aaa$ map the character $\chi\in X_\phi$ to the character $\chi^\aaa=\chi\circ(\aaa^{-1}*\phi)_\aaa\in X_{\aaa^{-1}*\phi}$, where by $\aaa^{-1}*\phi$ we, of course, mean the unique  $\sgn$-normalized Drinfeld module such that $\aaa*(\aaa^{-1}*\phi)=\phi$. This is a semigroup action. There is also an action of $\Gal(\K/K)$ on $X$ given by $g\chi= \chi\circ g\in X_{g^{-1}(\phi)}$, where $g^{-1}(\phi)$ is the Drinfeld module defined by $g^{-1}(\phi)_a=g^{-1}(\phi_a)$. This action commutes with the ideal action.

\smallskip

The action of $I$ on $X$ defines a partially defined action of the group of fractional ideals of $\OO$ on~$X$, which gives rise to a transformation groupoid $\GG_J$. The C$^*$-algebra $C_{K,\infty}$ underlying Jacob's system is the C$^*$-algebra of this groupoid. Alternatively, one can say that $C_{K,\infty}$ is the semigroup crossed product $C(X)\rtimes I$ with respect to the action
$$
(\aaa f)(\chi)=\begin{cases}f(\chi'),&\text{if}\ {\chi'}^\aaa=\chi,\\
0,&\text{if no such}\ \chi'\ \text{exists.}\end{cases}
$$
The one-parameter family $\aaa\mapsto N(\aaa)^{it}$ of characters of $I$ defines a one-parameter group of automorphisms $\sigma_t$ of $C_{K,\infty}$.

\smallskip

We want to show that the system $(C_{K,\infty},\sigma)$ is isomorphic to the system $(A_{\K,\{\infty\}},\sigma)$ introduced in Section~\ref{sphase}. The latter system was defined using the action of $I=\D^+_{\{\infty\}}$ on $Y_{\K,\{\infty\}}=\Gal(\K/K)\times_{\ohs}\hat\OO$. Define an action of $\Gal(\K/K)$ on $Y_{\K,\{\infty\}}$ by $g(x,y)=(g^{-1}x,y)$.

\begin{theorem} \label{tJ}
There is a $\Gal(\K/K)$- and $I$-equivariant homeomorphism $\pi\colon X\to Y_{\K,\{\infty\}}$. In particular, the C$^*$-dynamical systems $(C_{K,\infty},\sigma)$ and $(A_{\K,\{\infty\}},\sigma)$ are isomorphic.
\end{theorem}

\bp
Fix a $\sgn$-normalized rank-one Drinfeld module $\phi^0$. Recall that the semigroup $\PP^+$ of principal ideals with positive generators acts trivially on $\phi^0$, so we have an action of $\PP^+$ on $X_{\phi^0}$. The semigroup~$\PP^+$ can be identified with $\OO^\times_+$. The latter semigroup acts on $\hat\OO$ by multiplication. Let us show first that there exists a $\PP^+$-equivariant continuous isomorphism $\pi^0\colon X_{\phi^0}\to\hat\OO$.

The $\OO$-module $\phi^0(\C_\infty)^{\tor}$ is isomorphic to~$K/\OO$. Indeed, if $\phi^0$ is defined by a lattice $\Lambda$, then $\phi^0(\C_\infty)^{\tor}\cong K\Lambda/\Lambda$. The lattice $\Lambda$ has the form~$\xi\aaa$ for some $\xi\in\C^\times_\infty$ and $\aaa\subset\OO$. Then $K\Lambda/\Lambda\cong K/\aaa$. Next, the closure of $\aaa$ in $\hat\OO$ has the form~$g\hat\OO$ for some $g\in\akf^*\cap\hat\OO$. Therefore
$$
K/\aaa=\akf/g\hat\OO\cong\akf/\hat\OO=K/\OO,
$$
and hence $\phi^0(\C_\infty)^{\tor}\cong K/\OO$.

It is well-known that the additive group $\akf$ is self-dual. Furthermore, the pairing on $\akf\times\akf$ can be defined as $(a,b)\mapsto\omega(ab)$ for a character $\omega$, and with an appropriate choice of $\omega$ the annihilator of $\hat\OO$ is $\hat\OO$. It follows that there exists an $\OO$-module isomorphism
$$
\widehat{K/\OO}=\widehat{\akf/\hat\OO}\cong\hat\OO,
$$
where the $\OO$-module structure on $\widehat{K/\OO}$ is defined by $a\chi=\chi(a\,\cdot)$.

Combining the isomorphisms $\phi^0(\C_\infty)^{\tor}\cong K/\OO$ and $\widehat{K/\OO}\cong\hat\OO$, we get the required $\OO$-module isomorphism
$$
\pi^0\colon X_{\phi^0}=\widehat{\phi^0(\C_\infty)^{\tor}}\to\hat\OO,
$$
where the $\OO$-module structure on $\widehat{\phi^0(\C_\infty)^{\tor}}$ is defined by $a\chi=\chi\circ\phi^0_a$. Since $\phi^0_{(a)}=\phi^0_a$ for positive~$a$, this isomorphism is $\PP^+$-equivariant.

\smallskip

Consider now the extension $H^+/K$ defined in the previous section. As we observed there, the Artin map $\akf^*\to\Gal(\K/K)$ gives an isomorphism of $\ohs$ onto $\Gal(\K/H^+)$, so the open subset $Y^0=\Gal(\K/H^+)\times_{\ohs}\hat\OO\subset Y_{\K,\{\infty\}}$ can be identified with $\hat\OO$. Hence the map $\pi^0$ can be considered to be a homeomorphism of $X_{\phi^0}$ onto $Y^0$.

We want to extend $\pi^0$ to a map $\pi\colon X\to Y_{\K,\{\infty\}}$. Let $\phi$ be a $\sgn$-normalized rank-one Drinfeld module, and $\chi\in X_\phi$ be a character. There exists an ideal $\aaa$ such that $\phi=\aaa*\phi^0$. Then $\chi^\aaa=\chi\circ\phi_\aaa^0\in X_{\phi^0}$. By definition of $\phi^0_\aaa$ the kernel of $\phi^0_\aaa\colon \phi^0(\C_\infty)^{\tor}\to\phi(\C_\infty)^{\tor}$ is exactly $\phi^0[\aaa]$. Therefore the kernel of~$\chi^\aaa$ contains $\phi^0[\aaa]$. Hence, under our isomorphism of $\phi^0(\C_\infty)^{\tor}$ with $K/\OO$ the kernel of $\chi^\aaa$ contains $\aaa^{-1}\OO/\OO$. Since the annihilator of $\aaa^{-1}\hat\OO$ in $\akf$ is $\aaa\hat\OO$, we conclude that $\pi^0(\chi^\aaa)\in\aaa\hat\OO$. In particular, $\pi^0(\chi^\aaa)\in\aaa Y_{\K,\{\infty\}}$, so we can define
$$\pi(\chi)=\aaa^{-1}\pi^0(\chi^\aaa)\in Y_{\K,\{\infty\}}.$$
Since $\pi^0$ is $\pplus$-equivariant, this definition does not depend on the choice of $\aaa$ (such that $\phi=\aaa*\phi^0$).

We have therefore extended $\pi^0$ to a continuous map $\pi\colon X\to Y_{\K,\{\infty\}}$. By construction this map is equivariant with respect to the ideal action.

\smallskip

Next let us show that $\pi$ is $\Gal(\K/K)$-equivariant. For every nonzero proper ideal $\mm\subset\OO$ consider the finite sets
$$
X_\mm=\bigsqcup_\phi\widehat{\phi[\mm]}\ \ \text{and}\ \ Y_\mm=\Gal(K_\mm/K)\times_{\ohs}(\hat\OO/\mm\hat\OO)=\Gal(K_\mm/K)\times_{(\OO/\mm)^*}(\OO/\mm).
$$
Similarly to $X$ and $Y_{\K,\{\infty\}}$ these sets carry actions of $\Gal(\K/K)$ and $I$; note, however, that if $\aaa$ is not prime to $\mm$, then the actions by $\aaa$ are defined by non-injective maps. Our isomorphism $\widehat{\phi^0(\C_\infty)^{\tor}}\simeq \hat\OO$ induces an isomorphism $\widehat{\phi^0[\mm]}\cong\hat\OO/\mm\hat\OO=\OO/\mm$. Since $\Gal(K_\mm/H^+)\simeq(\OO/\mm)^*$, we can identify $Y^0_\mm=\Gal(K_\mm/H^+)\times_{(\OO/\mm)^*}(\OO/\mm)\subset Y_\mm$ with $\OO/\mm$. Hence, similarly to the construction of $\pi$, by choosing for every $\phi$ an ideal $\aaa_\phi\in\imo$ such that $\phi=\aaa_\phi*\phi^0$ we can extend the isomorphism $\widehat{\phi^0[\mm]}\cong\OO/\mm$ to a map $\pi_\mm\colon X_\mm\to Y_\mm$. Clearly, the diagram
$$
\xymatrix{X\ar[r]^\pi\ar[d] & Y\ar[d]\\ X_\mm\ar[r]_{\pi_\mm} & Y_\mm},
$$
where the vertical arrows are the obvious quotient maps, is commutative. It follows that $\pi_\mm$ is independent of the choice of $\aaa_\phi$ and is $I$-equivariant.

Take a $\sgn$-normalized rank-one Drinfeld module $\phi$,  a character $\chi\in \widehat{\phi[\mm]}$, and $\aaa\in\imo$. For $\lambda\in(\aaa^{-1}*\phi)[\mm]=\sigma_\aaa^{-1}(\phi)[\mm]$ we have that
$$\sigma_\aaa(\lambda)=(\aaa^{-1}*\phi)_\aaa(\lambda),$$
so $\sigma_\aaa\chi=\chi\circ\sigma_\aaa=\chi\circ(\aaa^{-1}*\phi)_\aaa=\chi^\aaa$. Hence the Galois action of $\Gal(K_\mm/K)$ on $X_\mm$ corresponds to the ideal action of $\imo$ via the Artin map.

Under the map $\pi_\mm$ above, this Galois action is then transported to $Y_\mm$ by
$$\pi_\mm(\sigma_\aaa\chi)=\pi_\mm(\chi^\aaa)=\aaa\pi_\mm(\chi).$$

Recall now that the action of $\aaa$ on $Y_\mm$ is defined using the action $g(x,y)=(xr_{K_\mm/K}(g)^{-1},gy)$ on $\Gal(K_\mm/K)\times(\hat\OO/\mm\hat\OO)$, where $g\in\akf^*\cap\hat\OO$ is any element such that $\aaa=g\hat\OO\cap\OO$. Since $\aaa$ is prime to $\mm$, we can take $g$ such that $g_\pp=1$ for all primes $\pp$ dividing $\mm$. Then the action of $g$ on $\hat\OO/\mm\hat\OO$ is trivial and $r_{K_\mm/K}(g)=\sigma_\aaa$. Therefore, $\aaa\pi_\mm(\chi)=\sigma_\aaa\pi_\mm(\chi)$, so that
$$\pi_\mm(\sigma_\aaa\chi)=\sigma_\aaa\pi_\mm(\chi).$$
Since the Artin map $\imo\to\Gal(K_\mm/K)$ is surjective, we conclude that $\pi_\mm$ is $\Gal(K_\mm/K)$-equivariant.

Now note that $\displaystyle X=\lim_{\leftarrow}X_\mm$ and $\displaystyle Y=\lim_{\leftarrow}Y_\mm$. Hence the $\Gal(\K/K)$-equivariance of $\pi\colon X\to Y$ follows from the $\Gal(K_\mm/K)$-equivariance of $\pi_\mm$.

\smallskip

It remains to show that $\pi$ is a homeomorphism. The space $X$ is the disjoint union of the open sets~$g X_{\phi^0}$, where $g$ runs over representatives of $\Gal(\K/K)/\Gal(\K/H^+)\cong\Gal(H^+/K)\cong\Pic^+(\OO)$. Similarly, $Y_{\K,\{\infty\}}$ is the disjoint union of the sets $g Y^0$. Since $\pi$ is $\Gal(\K/K)$-equivariant and defines a homeomorphism of $X_{\phi^0}$ onto $Y^0$, we conclude that $\pi$ is a homeomorphism.
\ep

\begin{remark}
The map $\pi$ depends on the choice of a $\sgn$-normalized rank-one Drinfeld module $\phi^0$, an $\OO$-module isomorphism  $\phi^0(\C_\infty)^{\tor}\cong K/\OO$, and a character $\omega$ of $\akf$ defining a pairing on $\akf\times\akf$ such that $\hat\OO^\perp=\hat\OO$. It is not difficult to see that if $\pi'\colon X\to Y_{\K,\{\infty\}}$ is another $\Gal(\K/K)$- and $I$-equivariant homeomorphism, e.g.~one constructed using a different choice of the above three ingredients, then $\pi'(\chi)=g\pi(\chi)$ for a uniquely defined $g\in\Gal(\K/K)$.
\end{remark}

Applied to the system $(A_{\K,\{\infty\}},\sigma)$, our Theorem~\ref{tphase} summarizes \cite[Theorems~4.3.10,~4.4.15]{Ja}. Furthermore, by Theorem~\ref{ttype} and Corollary~\ref{cconstantf}, the type of the unique KMS$_\beta$-state of Jacob's system for $\beta\in(0,1]$ is III$_{q^{-\beta d_\infty}}$. This corrects a mistake in \cite[Theorem~4.5.8]{Ja}, which asserts that the type is III$_{q^{-\beta}}$, a mistake partially caused by a wrong formulation of the Chebotarev density theorem.\footnote{The computation in \cite{Ja} relies on this theorem, but in a way that is different from ours. Unfortunately, the strategy in \cite{Ja} does not work even for $d_\infty=1$, when the formulation of the Chebotarev density theorem becomes correct. The mistake is in the proof of crucial Lemma~4.5.5, the same lemma where the Chebotarev density theorem is used, which does not take into account that the elements $\mu_\pp$ and $\mu_\qq$ do not belong to $M[\dd]$. In fact, the assertion of that lemma is not correct, it can be shown that already the center of $M[\dd]$ has elements that are not $\Gal(K_\dd/k)$-invariant.}

\medskip

Another approach to defining a Bost-Connes system for function fields is that of Consani and Marcolli \cite{CM}. Their setting is different, as they develop a theory of dynamical systems for algebras of $\C_\infty$-valued functions. However, these algebras arise from groupoids, which makes it natural to consider the relationship between these groupoids and the ones we consider in this paper.

\smallskip

A one-dimensional $K$-lattice in $K_\infty$ is a pair $(\Lambda,\varphi)$, where $\Lambda\subset K_\infty$ is a rank-one lattice and $\varphi\colon K/\OO\to K\Lambda/\Lambda$ is an $\OO$-module map. Two one-dimensional $K$-lattices $(\Lambda_1,\varphi_1)$ and $(\Lambda_2,\varphi_2)$ are called commensurable, if $\Lambda_1$ and $\Lambda_2$ are commensurable (equivalently, $K\Lambda_1=K\Lambda_2$) and the maps $K/\OO\to K\Lambda_i/(\Lambda_1+\Lambda_2)$ defined by $\varphi_1$ and $\varphi_2$ coincide.

The one-dimensional $K$-lattices in $K_\infty$ can be parametrized (see \cite{CM},\cite{LLN}) by the set
$$K_\infty^\times\times_{K^\times}\akf^*\times_{\ohs}\hat\OO.$$
There is a partial action of $\A_{K,f}^*$ on this space given by
$$g(\xi,x,y)=(\xi,xg^{-1},gy)$$
as long as $gy\in\hat\OO$, and which is undefined if this is not the case. This action descends to a partial action of fractional ideals of $\OO$, and the corresponding orbit equivalence relation is exactly the relation of commensurability. The action defines a groupoid $\tilde\GG$ with the set of one-dimensional $K$-lattices in $K_\infty$ as its object space.

Let $\GG_{CM}$ be the quotient groupoid obtained by identifying elements $(\xi,x,y)$ and $(\zeta\xi,x,y)$ in the object space for $\zeta\in K_\infty^\times$. This is the main groupoid considered in \cite{CM}.
It has the object space
$$\GG_{CM}^0=K^\times\backslash\akf^*\times_{\ohs}\hat\OO
\simeq\Gal(K^{\ab,\infty}/K)\times_{\ohs}\hat\OO=Y_{K^{\ab,\infty},\{\infty\}}.$$
This identification respects the actions of the semigroup $I=\D^+_{\{\infty\}}$ of ideals, so the groupoid considered in \cite{CM} gives rise to the dynamical system $(A_{K^{\ab,\infty},\{\infty\}},\sigma)$. As for Jacob's system above, the type of the unique KMS$_\beta$-state for this system for $\beta\in (0,1]$ is III$_{q^{-\beta d_\infty}}$.

\smallskip

Theorem~\ref{tJ} clarifies the relation between the groupoids $\GG_{CM}$ and $\GG_J$ of Consani-Marcolli and Jacob: $\GG_{CM}$ is isomorphic to the quotient of $\GG_J$ by the action of $\Gal(\K/K^{\ab,\infty})\cong K^\times/K^\times_+\cong\F^\times_\infty$. Also note that in the case of the Bost-Connes system for $\Q$ we only divide out by scaling by positive reals. The natural analogue of the positive reals in our setting is the subgroup $K_\infty^+\subset K^\times_\infty$ of positive elements. The groupoid $K_\infty^+\backslash\tilde\GG$ has the object space
$$K_\infty^+\backslash K_\infty^\times\times_{K^\times}\akf^*\times_{\ohs}\hat\OO\simeq \F_\infty^\times\times_{K^\times}\akf^*\times_{\ohs}\hat\OO\simeq\Gal(\K/K)\times_{\ohs}\hat\OO,$$
and the ideal action is identical to the one on $Y_{\K,\{\infty\}}$. Therefore, the modification of the construction of Consani-Marcolli obtained by considering $K$-lattices in $K_\infty$ up to scaling by positive elements gives rise to a groupoid isomorphic to Jacob's groupoid $\GG_J$.

\medskip

We finish by reformulating the classification of KMS$_\beta$-states of $(A_{\K,\{\infty\}},\sigma)$ for $\beta>0$ in terms of relatively invariant measures on finite adeles. For number fields such a reformulation was used in~\cite{N2,N3}.

\begin{theorem}
Consider measures $\mu$ on $\akf$ such that
\begin{equation} \label{et10}
\mu(\hat\OO)=1\ \ \text{and}\ \ \mu(a\,\cdot)=N(a)^{-\beta}\mu\ \ \text{for all}\ \ a\in K^\times_+,
\end{equation}
where $N(a)=N(a\OO)=q^{-d_\infty\ord_\infty(a)}$. Then
\enu{i} for every $0<\beta\le1$ there exists a unique measure $\mu_\beta$ satisfying~\eqref{et10}; furthermore, the action of~$K^\times_+$ on $(\Z\times\akf,\lambda\times\mu_\beta)$ given by $a(n,x)=(\ord_{\infty}(a)+n,ax)$, where $\lambda$ is the counting measure, is ergodic;
\enu{ii} for every $\beta>1$ and $g\in\akf^*$ there exists a unique measure  satisfying \eqref{et10} that is concentrated on $K^\times_+g\subset\akf^*$; this way we get, for every $\beta>1$, a one-to-one correspondence between extremal measures satisfying \eqref{et10} and points of the compact group $\akf^*/K^\times_+\cong\Gal(\K/K)$.
\end{theorem}

\bp As we already remarked in the proof of Theorem~\ref{tphase}, classifying KMS$_\beta$-states of $(A_{\K,\{\infty\}},\sigma)$ for $\beta\ne0$ is equivalent to classifying measures $\nu$ on $X_{\K,\{\infty\}}=\Gal(\K/K)\times_{\ohs}\akf$ such that $\nu(Y_{\K,\{\infty\}})=1$ and $\nu(\aaa\,\cdot)=N(\aaa)^{-\beta}\nu$ for all $\aaa\in I$. Consider the subset $X^0=\Gal(\K/H^+)\times_\ohs\akf\subset X_{\K,\{\infty\}}$ and note that $X^0$ can be identified with $\akf$, since $\Gal(\K/H^+)\cong\ohs$. The set $X^0$ is $\PP^+$-invariant, $IX^0=X_{\K,\{\infty\}}$ and $\aaa X^0\cap X^0\ne\emptyset$ only for $\aaa\in\PP^+$. It is easy to see that this implies that any Radon measure $\mu$ on $X^0$ such that $\mu(\aaa\,\cdot)=N(\aaa)^{-\beta}\mu$ for all $\aaa\in\PP^+$ extends uniquely to a Radon measure $\nu$ on $X_{\K,\{\infty\}}$ such that $\nu(\aaa\,\cdot)=N(\aaa)^{-\beta}\nu$ for all $\aaa\in I$. It follows that classifying measures~$\nu$ on $X_{\K,\{\infty\}}$ as above is equivalent to classifying measures on $\akf$ satisfying~\eqref{et10}. Therefore all the statements of theorem, with the exception of the second part of (i), are indeed a reformulation of Theorem~\ref{tphase}.

Since the unique KMS$_\beta$-state of $(A_{\K,\{\infty\}},\sigma)$ for $\beta\in (0,1]$ is of type III$_{q^{-\beta d_\infty}}$, the action of $K^\times_+$ on $(\akf,\mu_\beta)$ is of type III$_{q^{-\beta d_\infty}}$. At the same time the norms of elements of $K^\times$ are integral powers of $q^{d_\infty}$. By Proposition~\ref{pergodic} we conclude that the action of the subgroup of elements of $K^\times_+$ of norm one, that is, of elements of order zero at infinity, on $(\akf,\mu_\beta)$ is ergodic. This is equivalent to the second statement in (i).
\ep

\bigskip

%%%%%%%%%%%%%%%%%%%%%%%%%%%%%%%%%%%%%%%%%%%%%%%%%%%%%%%%%%%%%%%%%%%%%%%%%%%%%%%%%%

\appendix

\section{Ratio set}\label{a1}

Let $(X,\mu)$ be a standard measure space with $\sigma$-finite measure $\mu$, and $\RR\subset X\times X$ be a non-singular countable measurable equivalence relation on $(X,\mu)$~\cite{FM1}. The non-singularity means that if $A\subset X$ is a set of measure zero, then the minimal $\RR$-invariant subset $X$ containing $A$ also has measure zero. Thanks to this assumption there is a measurable (with respect to a measure class on~$\RR$ whose projection onto $X$ coincides with $[\mu]$) map $c_\mu\colon\RR\to\R^*_+$ such that $c_\mu(x,z)=c_\mu(x,y)c_\mu(y,z)$ if $x\sim_\RR y\sim_\RR z$, and for any measurable bijective map $T\colon A\to B$ with graph in $\RR$ we have
$$
\frac{dT^{-1}\mu}{d\mu}(x)=c_\mu(x,Tx)\ \ \text{for}\ \ \mu\text{-a.e.}\ \ x\in A,
$$
where $T^{-1}\mu$ is the measure on $A$ defined by $(T^{-1}\mu)(Z)=\mu(TZ)$. The map $c_\mu$ is called the Radon-Nikodym cocycle.

By definition~\cite{kr,FM1} the ratio set $r(\RR,\mu)$ is the intersection of the essential ranges of the restrictions of $c_\mu$ to $\RR\cap(Z\times Z)$ for all measurable subsets $Z\subset X$ of positive measure. By the essential range we mean the smallest closed set which contains the value of almost every point.

Equivalently we can say that the ratio set is defined as follows.
For a measurable set $Z\subset X$ of positive measure and $\eps>0$ denote by $r_{Z,\eps}(\RR,\mu)$ the set of numbers $\lambda\ge0$ such that there exist measurable subsets $A,B\subset Z$ of positive measure and a measurable bijective map $T\colon A\to B$ with graph in $\RR$ such that
$|c_\mu(x,Tx)-\lambda|<\eps$ for all $x\in A$. Then
$$
r(\RR,\mu)=\bigcap_{\mu(Z)>0,\ \eps>0}r_{Z,\eps}(\RR,\mu).
$$
Denote by $\bar r_{Z,\eps}(\RR,\mu)$ the subset of $\lambda\in r_{Z,\eps}(\RR,\mu)$ for which we can choose $A$ and $B$ as above such that $\mu(Z\setminus(A\cup B))=0$. Using Zorn's lemma it is easy to see that $r(\RR,\mu)\subset \bar r_{Z,\eps}(\RR,\mu)$. Furthermore, since $c_\mu(x,x)=1$ for all $x$, for every $\lambda\in r(\RR,\mu)\setminus\{1\}$ we can choose the sets $A$ and $B$ such that $A\cap B=\emptyset$.

Clearly,  $r(\RR,\mu)$ is a closed subset of $[0,+\infty)$. Furthermore, $r(\RR,\mu)\setminus\{0\}$ is a subgroup of $\R^*_+$. If we replace~$\mu$ by an equivalent measure, the ratio set remains the same.

\begin{lemma} \label{lapp}
For every $\lambda>0$, there exist $\eps_0>0$ and $\delta_0>0$ such that if $Z$ and $Z'$ are measurable subsets of $X$ of finite positive measure, $0<\eps<\eps_0$, $0<\delta<\delta_0$, $\lambda\in\bar r_{Z',\eps}(\RR,\mu)$ and $\mu(Z\Delta Z')<\delta\mu(Z')$, then $\lambda\in r_{Z,\eps}(\RR,\mu)$.
\end{lemma}

\bp Let $A,B\subset Z'$ and $T\colon A\to B$ be as required in the definition of $\bar r_{Z',\eps}(\RR,\mu)$. If
$$
\mu(Z\cap A\cap T^{-1}(Z\cap B))>0,
$$
then clearly $\lambda\in r_{Z,\eps}(\RR,\mu)$. So assume that $Z\cap A\cap T^{-1}(Z\cap B)$ has measure zero. Then
\begin{equation} \label{eratio0}
\mu(Z\cap A)+\mu(T^{-1}(Z\cap B))\le\mu(A).
\end{equation}
For every measurable subset $\Omega\subset A$ we have
\begin{equation} \label{eratio1}
|\mu(T\Omega)-\lambda \mu(\Omega)|\le\eps\mu(\Omega).
\end{equation}
In particular,
$$
|\mu(Z\cap B)-\lambda \mu(T^{-1}(Z\cap B))|\le \eps\mu(T^{-1}(Z\cap B))\le\eps\mu(A).
$$
Multiplying \eqref{eratio0} by $\lambda$ we then get
$$
\lambda\mu(Z\cap A)+\mu(Z\cap B)\le(\lambda+\eps)\mu(A).
$$
Since $A,B\subset Z'$ and $\mu(Z\Delta Z')<\delta\mu(Z')$, the above inequality yields
$$
\lambda \mu(A)+ \mu(B)\le (\lambda+\eps)\mu(A)+\delta(\lambda+1)\mu(Z'),
$$
hence
\begin{equation} \label{eratio3}
\mu(B)\le \eps \mu(A)+\delta(\lambda+1)\mu(Z')\le(\eps+\delta(\lambda+1))\mu(Z').
\end{equation}

On the other hand, from \eqref{eratio0} we also get
$$
|\mu(B)-\lambda \mu(A)|\le\eps \mu(A).
$$
Hence $\mu(B)\ge(\lambda-\eps)\mu(A)$ and
$$
\mu(Z')\le \mu(A)+\mu(B)\le \mu(A)+(\lambda-\eps)\mu(A),
$$
so that
$$
\mu(B)\ge\frac{\lambda-\eps}{1+\lambda+\eps}\mu(Z').
$$
This inequality contradicts \eqref{eratio3}, if $\eps+\delta(\lambda+1)<({\lambda-\eps})/({1+\lambda+\eps})$.
\ep

\begin{proposition} \label{pproduct}
Let $\RR$ be a non-singular countable measurable equivalence relation on a standard measure space $(X,\mu)$. Assume $(Y,\nu)$ is another standard measure space, and define an equivalence relation $\RR\times\id$ on $X\times Y$ such that $(x_1,y_1)\sim_{\RR\times\id}(x_2,y_2)$ if and only if $x_1\sim_\RR x_2$ and $y_1=y_2$. Then $$r(\RR\times\id,\mu\times\nu)\setminus\{0\}=r(\RR,\mu)\setminus\{0\}.$$
\end{proposition}

\bp The inclusion $r(\RR\times\id,\mu\times\nu)\subset r(\RR,\mu)$ is obvious. Take $\lambda\in r(\RR,\mu)\setminus\{0\}$ and $\eps>0$. Since for every measurable set $U\subset X$ of positive measure we have $\lambda\in \bar r_{U,\eps}(\RR,\mu)$, we get  $\lambda\in \bar r_{Z,\eps}(\RR\times\id,\mu\times\nu)$ for every measurable  rectangular set $Z=U\times V$ of positive measure. But then the same is true for every set $Z$ of positive measure in the algebra generated by measurable rectangular sets. Since every measurable set in $X\times Y$ can be approximated by sets in this algebra, by the previous lemma we conclude that $\lambda\in r_{Z,\eps}(\RR\times\id,\mu\times\nu)$ for every measurable set $Z$ of positive measure.
\ep

Another related immediate consequence of Lemma~\ref{lapp} is the following continuity result.

\begin{proposition} \label{psequence}
Let $\RR$ be the orbit equivalence relation defined by an action of a countable group~$\Gamma$ on a standard measure space $(X,\mu)$ by non-singular transformations. Assume $\{\xi_n\}^\infty_{n=1}$ is an increasing sequence of $\Gamma$-invariant measurable partitions such that $\vee_n\xi_n$ is the partition into points. Also assume that the measure $\mu_n$ defined by~$\mu$ on $X_n=X/\xi_n$ is $\sigma$-finite and the Radon-Nikodym cocycle~$c_\mu$ is $\xi_n$-measurable. Let $\RR_n$ be the orbit equivalence relation defined by the action of $\Gamma$ on $X_n$. Then
$$
r(\RR,\mu)\setminus\{0\}=\bigcap_nr(\RR_n,\mu_n)\setminus\{0\}.
$$
\end{proposition}

Now assume that we have an ergodic countable measurable equivalence relation $\RR$ on $(X,\mu)$, meaning that for every $\RR$-invariant measurable set either the set itself or its complement has measure zero. Since $r(\RR,\mu)\setminus\{0\}$ is a closed subgroup of $\R^*_+$, we have the following possibilities for the ratio set: $\{1\}$, $\{0,1\}$, $\{0\}\cup\{\lambda^n\mid n\in\Z\}$ for some $\lambda\in(0,1)$, and $[0,+\infty)$. In the last three cases the equivalence relation is said to be of
type III$_0$, III$_\lambda$, and III$_1$, respectively.

\smallskip

If $\RR$ is ergodic, it is easy to see that for every set $Z\subset X$ of positive measure the ratio set of the equivalence relation on $Z$ induced by $\RR$ is the same as the ratio set of $\RR$.

\smallskip

Denote by $\RR^0$ the kernel of $c_\mu$, that is, the set
of points $(x,y)\in\RR$ such that $c_\mu(x,y)=1$. It is a measurable equivalence relation, and the measure $\mu$ is $\RR^0$-invariant. The next proposition follows from a general result on type III factors, see~\cite[Corollary~29.12]{Stra}, but it is much more elementary.

\begin{proposition} \label{pergodic}
Let $\RR$ be an ergodic non-singular countable measurable equivalence relation on a standard measure space $(X,\mu)$. Assume $\RR^0$ is ergodic. Then $r(\RR,\mu)$ coincides with the essential range of $c_\mu$.

Conversely, assume $r(\RR,\mu)$ coincides with the essential range of $c_\mu$. Also assume that $r(\RR,\mu)\setminus\{0\}$ is discrete, so $\RR$ is not of type III$_1$. Then $\RR^0$ is ergodic.
\end{proposition}

\bp Let us prove the second statement, since it is the only one used in the main text.

Suppose $\RR^0$ is not ergodic, so there exists an $\RR^0$-invariant measurable subset $A\subset X$ such that both $A$ and $A^c=X\setminus A$ have positive measure. By ergodicity of $\RR$, there exist subsets $A_1\subset A$ and $B_1\subset A^c$ of positive measure and a measurable bijective map $T\colon A_1\to B_1$ with graph in $\RR$. Since the nonzero essential values of $c_\mu$ form a discrete set, by replacing~$A_1$ by a smaller set we may assume that there exists $\lambda>0$ such that $c_\mu(x,Tx)=\lambda$ for all $x\in A_1$. Since~$\lambda^{-1}$ lies in the ratio set and it is isolated in the set of essential values of $c_\mu$, we can find subsets $B\subset B_1$ and $C\subset B_1$ of positive measure and a measurable bijective map $S\colon B\to C$ with graph in $\RR$ such that $c_\mu(x,Sx)=\lambda^{-1}$ for all $x\in B$. Put $A_2=T^{-1}B$. Then $A_2$ is a set of positive measure, $A_2\subset A$, $STA_2\subset A^c$ and $c_\mu(x,STx)=c_\mu(x,Tx)c_\mu(Tx,STx)=1$ for all $x\in A_2$. This contradicts the $\RR^0$-invariance of $A$.
\ep

Now consider a particular class of equivalence relations.
Let $\{(X_n,\mu_n)\}_{n=1}^\infty$ be a sequence of at most countable
probability spaces. Put $(X,\mu)=\prod_n(X_n,\mu_n)$, and define an
equivalence relation $\RR$ on $X$ by
$$
x\sim y\ \ \text{if}\ \ x_n=y_n\ \ \text{for all} \ \ n
\ \ \text{large enough}.
$$
This equivalence relation is non-singular and ergodic.

For a finite subset $I\subset\N$ and $a\in\prod_{n\in I}X_n$ put
$$
Z(a)=\{x\in X\mid x_n=a_n\ \text{for}\ n\in I\}.
$$
The asymptotic ratio set $r_\infty(\RR,\mu)$ consists by
definition~\cite{AW} of all numbers $\lambda\ge0$ such that for any
$\eps>0$ there exist a sequence $\{I_n\}^\infty_{n=1}$ of mutually
disjoint finite subsets of $\N$, disjoint subsets
$K_n,L_n\subset\prod_{k\in I_n}X_k$ and bijections $\varphi_n\colon
K_n\to L_n$ such that
$$
\left|\frac{\mu(Z(\varphi_n(a)))}{\mu(Z(a))}-\lambda\right|<\eps\ \text{for all}\
a\in K_n\ \text{and}\ n\ge1,\ \ \text{and}\ \ \sum_{n=1}^\infty\sum_{a\in K_n}\mu(Z(a))=\infty.
$$
Clearly, the asymptotic ratio set is a closed subset of $[0,+\infty)$.

\begin{proposition}[\cite{kr}]
We have $r_\infty(\RR,\mu)\setminus\{0\}=r(\RR,\mu)\setminus\{0\}$.
\end{proposition}

The nontrivial part here is that $1\in r_\infty(\RR,\mu)$. What is not difficult to see, and this is more than enough for our purposes, is that
$$
r(\RR,\mu)\setminus\{1\}\subset r_\infty(\RR,\mu)\ \ \text{and}\ \ r_\infty(\RR,\mu)\setminus\{0\}\subset r(\RR,\mu).
$$
Let us sketch a proof of these inclusions. First of all note that the Radon-Nikodym cocycle is given~by
$$
c_\mu(x,y)=\prod_n\frac{\mu_n(y_n)}{\mu_n(x_n)}.
$$

Take $\lambda\in r(\RR,\mu)\setminus\{1\}$ and $\eps>0$. We can find measurable sets $A$ and $B$ and a measurable bijective map $T\colon A\to B$ with graph in $\RR$ such that $\mu(X\setminus(A\cup B))=0$ and $|c_\mu(x,Tx)-\lambda|<\eps$ for all $x\in A$. As we already noted, since $\lambda\ne1$, we can choose $A$ and $B$ to be disjoint. Then the measure of~$A$ is close to $(1+\lambda)^{-1}$. There exist $m\ge1$ and disjoint measurable subsets $A_1,\dots,A_N$ of~$A$ such that $\cup_iA_i$ is close to $A$ and the restriction of $T$ to $A_i$ is such that $Tx$ is obtained from $x\in A_i$ by applying a transformation $\varphi_{1i}$ to the first $m$ coordinates of $x$. Approximate the sets $A_i$ by unions of cylindrical sets $Z(a)$, $a\in K_1\subset\prod_{n\le n_1}X_n$, where $n_1\ge m$. We may assume that if $Z(a)$ is used in the approximation of $A_i$, then the ratio $\mu(Z(a)\cap A_i)/\mu(Z(a))$ is close to $1$. In particular, for every $a\in K_1$ the cylindrical set $Z(a)$ is used in approximating only one of the sets $A_i$. Then the transformations $\varphi_{1i}$ define an injective map $\varphi_1\colon K_1\to\prod_{n\le n_1}X_n$ such that $K_1\cap\varphi_1(K_1)=\emptyset$. By construction $\sum_{a\in K_1}\mu(Z(a))$ is close to $(1+\lambda)^{-1}$.

Next replace $X$ by $\prod_{n>n_1}X_n$ and repeat the same argument. Note that the ratio set for the new product-space remains the same, since $\RR$ is ergodic and $\prod_{n>n_1}X_n$ can be identified with the cylindrical subset $Z(a)$ of~$X$ for some $a\in\prod_{n\le n_1}X_n$. Continuing this process we gradually construct the required sets $K_n$ and maps $\varphi_n$ and conclude that $\lambda\in r_\infty(\RR,\mu)$.

Now take $\lambda\in r_\infty(\RR,\mu)\setminus\{0\}$. Then for every $\eps>0$ we can find measurable sets $A_n$ and $B_n$ and measurable bijective maps $T_n\colon A_n\to B_n$ with graphs in $\RR$ such that $|c_\mu(x,T_nx)-\lambda|<\eps$ for all $x\in A_n$, $A_n\cap B_n=\emptyset$, $T_n$~maps $A_n\cap A_m$ onto $B_n\cap A_m$ and $A_n\cap B_m$ onto $B_n\cap B_m$,  the sets $A_n\cup B_n$ are mutually independent and $\sum_n\mu(A_n)=\infty$. Consider the sets
$$
A=\cup_n\left(A_n\setminus\cup_{m<n}(A_m\cup B_m)\right)\ \ \text{and}\ \ B=\cup_n\left(B_n\setminus\cup_{m<n}(A_m\cup B_m)\right).
$$
Then the maps $T_n$ define a measurable bijective map $T\colon A\to B$, $A\cap B=\emptyset$ and $\mu(A\cup B)=1$. Therefore $\lambda\in \bar r_{X,\eps}(\RR,\mu)$. Since the same is true if we replace $X$ by $\prod_{n\ge m}X_n$, we conclude that $\lambda\in \bar r_{Z,\eps}(\RR,\mu)$ for any $Z$ which is a finite union of cylindrical sets $Z(a)$. By Lemma~\ref{lapp} it follows that $\lambda\in r_{Z,\eps}(\RR,\mu)$ for any measurable set $Z$ of positive measure. Hence $\lambda\in r(\RR,\mu)$.

\end{document}